\UseRawInputEncoding
\documentclass{amsart}
\usepackage{amsfonts,amssymb,amsmath,amsthm}
\usepackage{url}
\usepackage{enumerate}
\usepackage[all]{xy}

\newtheorem{thrm}{Theorem}[section]
\newtheorem{lem}[thrm]{Lemma}
\newtheorem{prop}[thrm]{Proposition}
\newtheorem{cor}[thrm]{Corollary}
\theoremstyle{definition}
\newtheorem{definition}[thrm]{Definition}
\newtheorem{question}[thrm]{Question}

\newtheorem{remark}[thrm]{Remark}

\newtheorem{example}[thrm]{Example}
\newtheorem{nt}[thrm]{Notation}
\newtheorem{discussion}[thrm]{Discussion}
\numberwithin{equation}{section}

\makeatletter
\@namedef{subjclassname@2020}{%
  \textup{2020} Mathematics Subject Classification}
\makeatother

\newcommand{\Ass}{\operatorname{Ass}}

\newcommand{\Spec}{\operatorname{Spec}}

\newcommand{\ara}{\operatorname{ara}}

\newcommand{\cd}{\operatorname{cd}}

\newcommand{\zd}{\operatorname{zd}}
\newcommand{\Ht}{\operatorname{ht}}
\newcommand{\pd}{\operatorname{pd}}

\newcommand{\Ext}{\operatorname{Ext}}

\newcommand{\Tor}{\operatorname{Tor}}

\newcommand{\Hom}{\operatorname{Hom}}

\newcommand{\depth}{\operatorname{depth}}

\newcommand{\vpl}{\operatornamewithlimits{\varprojlim}}

\newcommand{\vil}{\operatornamewithlimits{\varinjlim}}

\newcommand{\fm}{\mathfrak{m}}

\newcommand{\fp}{\mathfrak{p}}

\newcommand{\fa}{\mathfrak{a}}
\newcommand{\fb}{\mathfrak{b}}
\newcommand{\fc}{\mathfrak{c}}

\newcommand{\fn}{\mathfrak{n}}

\author[A.\,F.\,Boix]{Alberto F.\,Boix}

\address{Department of Applied Mathematics, Science, Materials Engineering and Electronic Technology, Universidad Rey Juan Carlos, C/Tulip\'an, s/n, 28933, M\'ostoles, Spain.}
\email{alberto.fboix@urjc.es}
\email{albertof.boix@gmail.com}

\author{Majid Eghbali}

\thanks{A. F. Boix was partially supported by the Spanish Ministry of Science and Innovation grant PID2019--104844GB--I00.}

\address{Department of Mathematics, Tafresh University, Tafresh 39518 79611, Iran.} \email{m.eghbali@tafreshu.ac.ir} \email{m.eghbali@yahoo.com}



\begin{document}

\title[Local cohomology, set-theoretically C-M ]
{Vanishing of local cohomology and set-theoretically Cohen-Macaulay
ideals $$\text{To\ the\ soul\ of\ professor\ Wolfgang\ Vogel}$$}

\begin{abstract}
In this paper, first, we generalize a result of Peskine-Szpiro on
the relation between the cohomological dimension and projective
dimension. Then, we give conditions for the vanishing of local
cohomology from local to global and vice versa. Our final goal in
the present paper is to examine the set-theoretically Cohen-Macaulay
ideals to find some cohomological characterization of these kinds of
ideals.
\end{abstract}

 \maketitle

\section{INTRODUCTION}\label{introduction}

Throughout this paper, all rings are commutative and Noetherian with identity.
For an ideal $I$ of a local ring $(R,\fm)$, the local
cohomology modules $H^i_I(R)$ may be considered as the isomorphism
$H^i_I(R) = \vil \Ext^i_R(R/I^t,R) \text{\ for\ } i\geq 0$. Moreover, herefater $\sqrt{I}$ will denote the radical of the ideal $I.$ On the other hand, given $(R,\fm)$ a local ring, $E$ will denote the injective hull of the residue field of $R,$ and $(-)^{\vee}$ will denote the Matlis duality functor $\Hom_R (-,E).$% On the other hand, by $\lim$ we mean inverse limit, and by $\operatorname{colim}$ we mean direct limit.

%Fix a ring endomorphism $\varphi: R \rightarrow R$. It
%induces a natural $\varphi$ action on the local cohomology modules
%$\varphi_{\ast}: H^i_{I}(R) \rightarrow H^i_{\varphi(I)R}(R)$ via
%$\varphi (r) \varphi_{\ast} (\eta) =\varphi_{\ast} (r \eta)$, where
%$r \in R, \eta \in H^i_{I}(R), \sqrt{I}=\sqrt{\varphi(I)R}$ which is
%an endomorphism of the underlying Abelian group (details, including
%notation, are given in Section $2$).
%Inspired of \cite{PS} and \cite{Lyu06}, as a generalization
%of Frobenius action, the mentioned action is an effective tool in the study of local
%cohomology modules as it has been used recently in \cite{Al,B-E,EABR,S-W}.

In their landmark paper, Peskine and Szpiro \cite[Proposition 4.1]{PS} proved that
whenever $(R, \fm)$ is a regular local ring containing a field of positive characteristic and $I \subset R$ is
a Cohen-Macaulay ideal (i.e. the ring $R/I$ is Cohen-Macaulay),
$H^i_{I}(R)=0 \text{\ for\ all \ } i > \dim R-\depth R/I$. An
immediate implication of it is that the inequality
$\cd(R,I) \leq \pd R/I$ holds, where cohomological dimension $\cd(R,
I)$ resp. projective dimension $\pd R/I$ of $I$ is defined as
$$\cd(R, I) = \min \{i : H^j_I(R) = 0 \text{\ for\ all\ } j > i\},$$
resp.
$$\pd R/I=\sup \{ n|\ \Ext^i_R(R/I,N)=0,\ \text{for\ all\ } R-\text{mod.\ } N\ \text{and\ all\  }i \geq n+1 \}.$$

Let $\varphi: R\rightarrow R$ be a ring endomorphism, it induces a natural $\varphi$ action on the local cohomology modules $\varphi_{\ast}: H^i_{I}(R) \rightarrow H^i_{\varphi(I)R}(R)$ via $\varphi (r) \varphi_{\ast} (\eta) =\varphi_{\ast} (r \eta)$, where $r \in R, \eta \in H^i_{I}(R), \sqrt{I}=\sqrt{\varphi(I)R}$ which is an endomorphism of the underlying Abelian group (details, including notation, are given in Section \ref{notation and remarks}). Inspired of \cite{PS} and \cite{Lyu06}, as a generalization
of Frobenius action, the mentioned action is an effective tool in the study of local cohomology modules as it has been used recently in \cite{Al,B-E,EABR,S-W}.  

In Section \ref{on a result of Peskine-Szpiro}, (cf. Theorem \ref{PS Generalize}) we generalize the inequality $\cd (R,I)\leq\dim (R)-\depth R/I$ using the action of $\varphi$ on local cohomology (without any restriction on the characteristic of the ring).

Recall that $I$ is called a set-theoretically complete intersection ideal if there are $h=height (I):=\Ht (I)$ elements $g_1, \ldots , g_s \in R$ such that they generate an ideal which has the same radical as $I;$ in other words, $\Ht (I)= \ara (I),$ where $\ara (I)$ denotes the arithmetic rank of $I$. It is known \cite[3.3.4]{BroSha} that, in general, $\Ht (I) \leq \cd (R,I) \leq \ara (I)$, and, borrowing terminology from \cite{HellusSchenzel}, $I$ is called cohomologically complete intersection, whenever  $\Ht
(I) = \cd (R,I)$. Notice that a set-theoretically complete intersection ideal is a cohomologically complete intersection one, but the converse is no longer true.

From the other point of view, the aforementioned result of  Peskine and Szpiro says
 that $I$ is cohomologically complete intersection, whenever it is a perfect ideal.
  Recently,  Varbaro has shown that their idea works for all Noetherian rings
  of positive characteristic (cf. \cite[Corollary 2.2]{Var}).
   On the other hand, a conjecture of Hartshorne on the relation between being a set-theoretically complete intersection curve in $\mathbb{P}^3_k$ over a field $k$ of characteristic $0$ and Cohen-Macaulayness of its coordinate ring (see Discussion \ref{Hartshorne}), motivated us to give the following Theorem as our main result in Section \ref{on a result of Peskine-Szpiro}:

\begin{thrm} (cf. Theorem \ref{CM cond}) Let $(R,\fm)$ be a Gorenstein local ring, let $I\subseteq R$ be an ideal, and let $\varphi:\ R\rightarrow R$ be a ring endomorphism. Suppose that
\begin{enumerate}[(a)]
\item $\varphi$ has the going down property;
\item the ideals $\{\varphi^t (I)R\}_t$ form a descending chain cofinal wil $\{I^t\}_t.$
%\item $\varphi$ is local and $\sqrt{\varphi (\mathfrak{m})R}=\mathfrak{m}.$
 \end{enumerate}
 Then the following statements are equivalent.
\begin{itemize}
  \item[(1)] $R/I$ is a Cohen-Macaulay ring.
  \item[(2)] $\Ht(I)=\cd(R,I)$ and $\Ext^r_R(R/\varphi^{t}(I)R,R) \longrightarrow
  \Ext^r_R(R/\varphi^{t+1}(I)R,R)$ are injective of nonzero $R$-modules for all $t$ and $r:=\depth R/I$.
  \end{itemize}
\end{thrm}

%Note that an endomorphism $\varphi$ satisfying condition (c) in the above Theorem is called of "finite length". It was defined in \cite{MMS} to introduce and develop the so-called algebraic entropy.

Second part of the paper is devoted to the study of set-theoretically Cohen-Macaulay ideals. Our very concrete motivation is the following question asked in \cite{S-W2}.

\begin{question} Given an affine variety $V$, does $V$ support a Cohen-Macaulay scheme, i.e. whether there exists a Cohen-Macaulay ring $R$ such that $V$ is isomorphic to $\Spec R_{red}$?
\end{question}

Let $X$ be a Cohen-Macaulay scheme. It fails to be true that $X_{red},$ the corresponding
 reduced scheme, is Cohen-Macaulay. In general, Singh and Walther defined an ideal $I$ in a regular local ring to be
\textit{set-theoretically Cohen-Macaulay} if there exists an ideal
$J \subset R$ with $\sqrt{I}=\sqrt{J}$ such that the ring $R/J$ is
Cohen-Macaulay. It is evident that set-theoretically complete
intersection and Cohen-Macaulay radical ideals are set-theoretically Cohen-Macaulay but the converse is no longer true. It is a natural question to ask on the relation between
Cohen-Macaulay and set-theoretically Cohen-Macaulay ideals.

Our first result in this direction is Proposition \ref{equivalence},
which is a consequence of Theorem \ref{CM cond}. Under the assumptions given in Proposition \ref{equivalence}, 
the concepts of Cohen-Macaulayness and set-theoretically Cohen-Macaulayness are
the same.

Not so much is known about the cohomological characterization of
set-theoretically Cohen-Macaulay ideals. In Section \ref{set theoretically C-M ideals}, among other results we show
that under certain conditions,  set-theoretically Cohen-Macaulay
ideals are the same as cohomologically complete intersections.

\begin{thrm} (cf. Theorem \ref{setCM-setCI}) Let $(R,\fm)$ be a  regular local ring, let $I\subseteq R$ be an ideal, and let $\varphi:\ R\rightarrow R$ be a ring endomorphism. Suppose that $\varphi$ is flat, and there is an ideal $J\subseteq R$ with $\sqrt{I}=\sqrt{J}$ such that $R/J$ is Cohen-Macaulay and $\{\varphi^t (J)R\}_t$ is cofinal with respect to $\{I^t\}_t.$ Then, one has $\Ht(I)=\cd (R,I)$.
\end{thrm}

Whenever the action of $\varphi$ on $R/\sqrt{I}$ is pure and flat, the equivalence between the cohomologically complete intersection
and set-theoretically Cohen-Macaulay ideals in Theorem \ref{setCM-setCI} holds, see Corollary \ref{equi.setCM-setCI}.

Another way to study set-theoretically Cohen-Macaulayness of a variety is through
linkage theory. Roughly speaking, the linkage is the study of two
varieties where their union has nice properties. Through the use of this
concept one may consider a variety linked with the second one
which one understands better.

 We cite the following
paragraph from a fruitful paper of A. Martsinkovsky and J. R.
Strooker \cite{Ma-St} on the significance of the concept of linkage
theory.

\emph{ 'It goes back to the late 19th and early 20th century, when
M. Noether, Halphen, and Severi used it to study algebraic curves in
$\mathbb{P}^3$. Linkage allows to pass from a given curve to another
curve, related in a geometric way to the original one. Iterating the
procedure one obtains a whole series of curves in the same linkage
class. The usefulness of this technique is explained by two
observations: (a) certain properties of the curve are preserved
under linkage, and (b) the resulting curves may be simpler, and thus
easier to handle, than the original one.'}

In the same vein, Peskine and Szpiro in \cite[Proposition
1.3]{P-S74} proved that in a Gorenstein local ring if $I$ and $J$
are linked ideals and $I$ is a Cohen-Macaulay ideal then so is $J$.
The same result was proved by Schenzel in \cite[Corollary 3.3]{Sch82}
for Buchsbaum ideals. In this direction, in Corollary \ref{setlink}
we give conditions where the set-theoretically Cohen-Macaulay
property can be shared between two linked ideals.

\section{NOTATIONS AND REMARKS}\label{notation and remarks}

We start this section with the idea used by Singh and Walther in \cite{S-W}.

Let $A$ be a commutative Noetherian ring with a flat endomorphism
$\varphi: A \rightarrow A$, and let $\fa$ be an ideal of $A$. We
denote by $\varphi_{\ast}A$ the following $(A,A)$-bimodule: for any
$r, r_1, r_2 \in A$,
$$r_1.(\varphi_{\ast}r).r_2:=\varphi_{\ast}(\varphi (r_1)rr_2).$$

Let $\Phi$ be the functor on the category of $A$-modules with $\Phi(M)=\varphi_{\ast} A \otimes_A M$.
The iteration $\Phi^t$ is the functor
$$\Phi^t(M)=\varphi_{\ast} A \otimes_A \Phi^{t-1}M,\ t \geq 1,$$
where $\Phi^0$ is interpreted as the identity functor; the reader will easily note that the flatness of $\varphi$ is equivalent to the exactness of $\Phi$. At once, one can realize that $\Phi^t(M)=\varphi^t_{\ast} A \otimes_A M$.

Let us notice that $\Phi(A) \cong A$ given by $\varphi_{\ast}r' \otimes r \mapsto \varphi (r)r'$. Furthermore, if $M$ and $N$ are $A$-modules, then \cite[2.6.1]{S-W} there are natural isomorphisms
$$\Phi(\Ext^i_A(M,N)) \cong \Ext^i_A(\Phi(M), \Phi(N)),\ \text{\ for\ all\ } i \geq 0.$$
Assume, in addition, that the ideals $\{\varphi^t(\fa)R\}_{t\geq 0}$ form a descending chain cofinal
with the chain $\{\fa^t\}_{t\geq 0}$, where ${\fa}$ is an ideal of $A$. Under these assumptions, one can easily check \cite[page 291]{S-W} that
$$\Phi(H^i_{\fa}(A)) \cong H^i_{\varphi(\fa)}(A)\cong H^i_{\fa}(A),\ \text{\ for\ all\ } i \geq 0.$$

Now, let us fix our notations. Throughout this section, let $R$ be a  commutative Noetherian local
 ring and $\varphi:R \longrightarrow R$ be a local ring endomorphism. Suppose that $I$ is an ideal of
$R$ and $\{\varphi^t(I)R\}_{t \geq 0}$ is a decreasing chain of ideals cofinal with $\{I^t\}_{t \geq 0}$. Keeping these assumptions, we introduce the following notations.

\begin{nt} Let  $\varphi:R \longrightarrow R$ be a flat ring endomorphism (a ring map
  $R \longrightarrow S$ is called flat if $S$ is flat as an $R$-module). In this case, we say that triple $(R, I, \varphi)$ has (Flat) property.
\end{nt}

Next result can be found in \cite[Exercise 7.1]{Mat}.

\begin{lem} \label{Matsumura} Let $A,B$ be two rings. If $B$ is a faithfully flat $A$-algebra
 then for an $A$-module $M$ one has $B \otimes_A M$ is $B$-flat if and only if $M$ is $A$-flat.
\end{lem}

In Lemma \ref{Matsumura}, put $B:=\hat{R}$, the $\fm$-adic
completion of local ring $R$ and $M=A:=R$. Then for every flat ring
endomorphism $\varphi:R \longrightarrow R$, the induced ring
endomorphism $\hat{\varphi}:\hat{R} \longrightarrow \hat{R}$ is
flat.

\begin{definition}
For two rings $A$ and $B$, a ring map $f:A \rightarrow B$ has
  the going-down property if, for any two primes $\fp' \subsetneq \fp$ of
   $A$ and any prime $\mathcal{P}$ in $B$ with $f^{-1}(\mathcal{P})=\fp$,
    there is a prime $\mathcal{P'} \subsetneq \mathcal{P}$ of $B$ such that
    $f^{-1}(\mathcal{P'})=\fp'$.
\end{definition}

\begin{nt} Let  $\varphi:R \longrightarrow R$ be a ring endomorphism
satisfying the going down property. Suppose that $I$ is an ideal of
$R$ with finite projective
 dimension.  We say that the triple $(R, I, \varphi)$ has (GD) property.
\end{nt}

\begin{example} \label{flat} 
\begin{enumerate}[(i)]

\item Notice that every flat homomorphism implies the going down property \cite[(5.D) Theorem 4]{MatCA}.

\item In general, an extension $A \subset B$ of domains, with $A$ normal and $B$ integral over $A$ satisfies going down property, for a ring homomorphism  $f:A \longrightarrow B$ \cite[(5.E) Theorem 5]{MatCA}.

\item Let $R$ be a commutative Noetherian ring, and let $\xymatrix@1{R\ar[r]^-{\varphi}& R}$ be a ring homomorphism such that the induced map on $\Spec (R)$ is the identity. Then, $\varphi$ satisfies the Going Down property.

\item  Take into account that the Frobenius endomorphism of a regular local ring of positive characteristic and for any field $k$ and any integer $t \geq 2$, the $k$-linear endomorphism $\varphi (x_i)=x^t_i$ of $k[x_1, \ldots , x_n]$ are the prototypical examples of flat endomorphisms.

\item For a given regular local ring $R$, by \cite[Theorem 33.1]{Mat} the endomorphism $\varphi: R \longrightarrow R$ which is of finite length, is flat. A local homomorphism $f:(A, \fm) \longrightarrow (B, \fn)$ is of finite length, if the following equivalent conditions hold: a) $B/f(\fm)B$ is Artinian. b) The only prime ideal of $B$ which contracts to $\fm$ is $\fn$, see \cite[Definition 1]{MMS} and the references given therein for more details.
      
Let $R=k[\![x_1,x_2,x_3,x_4]\!]$ be a formal power series ring of $4$ variables over a field $k$ (of characteristic zero). Define $\varphi: R \longrightarrow R$ by $x_1 \mapsto x^3_1,\ x_2 \mapsto x^2_2,\ x_3 \mapsto x^2_3,\ x_4 \mapsto x^7_4$ and $I$ an ideal generated by $z$ and $\varphi^i(z),\ i \in \mathbb{Z}$ (orbit of $z$ under $\varphi$), where $z:=x^2_2+x^5_3$. It is clear that $\varphi(R)$ is a flat $R$-module and $\sqrt{I}=\sqrt{\varphi (I)R}.$% i.e. $\{I^t\}_t$ and $\{\varphi^t (I)R\}_t, t \in \mathbb{Z}$ are cofinal.

\item More generally, suppose that $R$ is a regular local ring, by \cite[Theorem 13.3]{AIM} every contracting endomorphism $\varphi: R \longrightarrow R$ with $\sqrt{\varphi (\fm)R}=\fm$ is flat. We say $\varphi: R \longrightarrow R$ is contracting if $\varphi^i (\fm)R \subseteq \fm^2$ for some integer $i$. We refer the interested reader to \cite{AIM} for more details.
\end{enumerate} 
\end{example}

\begin{nt}\label{Inj property} Suppose that $\Ext^i_R(R/\varphi^t(I)R,R) \rightarrow \Ext^i_R(R/\varphi^{t+1}(I)R,R),$
is injective for all $i \geq 0$ and all $t \geq 0$. We say $(R, I,
\varphi)$ has (Inj) property.
\end{nt}

\begin{definition} \label{pure}
A ring homomorphism  $f:A \rightarrow B$ is pure if the map $f
\otimes 1:A \otimes_A M \rightarrow B \otimes_A M$  is injective for
each $A$-module $M$.  If $A$  contains a field of prime characteristic $p$, then $A$ is $F$-pure if the Frobenius homomorphism $F:A \rightarrow A, a \mapsto a^p$ is pure, where $a \in A$.
\end{definition}

\begin{remark} \label{pure2} Note that by virtue of
 \cite[Lemma 6.2]{Ho-Ro} purity of $\varphi:R \longrightarrow R$
 implies the same property for $\bar{\varphi}:R/I \longrightarrow
 R/\varphi^t(I)$ for all positive integer $t$. It follows from \cite[Theorem
2.8]{S-W} that, when $R$ is a regular ring and the triple
$(R,I,\varphi)$ satisfies the (Flat) property, purity of
$\bar{\varphi}:R/I \longrightarrow R/I$
 implies that $(R,I,\varphi)$ has (Inj) property.

Let $R=k[x_1, \ldots , x_n]$ be a polynomial ring over a field $k$. There exists a $k$-linear endomorphism
 $\varphi:R \longrightarrow R$ with $\varphi (x_i)=x^t_i$ for $1 \leq i \leq n$ and positive integer $t$.
  If $I \subset R$ is an ideal generated by square-free monomials, then there exists a pure endomorphism
   $\bar{\varphi}:R/I \longrightarrow R/I$, (cf. \cite[Example 2.2]{S-W}).
\end{remark}

 \section{ON A RESULT OF PESKINE-SZPIRO}\label{on a result of Peskine-Szpiro}

We start by recalling properties of some homological invariants.
Then discuss and examine their relationships using ring
endomorphisms. Let $(R,\fm)$ be a local ring and $I$ be an ideal of
$R$ of finite projective dimension ($\pd R/I < \infty$). The
Auslander-Buchsbaum formula says that $\pd_R R/I=\depth R-\depth
R/I$. On the other hand, when $R$ is a complete local domain, the Hartshorne--Lichtenbaum Theorem gives a necessary and sufficient condition to guarantee $\cd (R,I)\leq\dim (R)-1,$ and Huneke and Lyubeznik in \cite[Theorem
1.1]{Hun-Lyu} give conditions for $\cd (R,I) \leq \dim R-2$, where
$R$ is a regular local ring containing a field. The interested reader may like to consult \cite{Bhattacharyya2021} and the references given therein for additional details.

Hereafter in this section, we focus on finding sufficient conditions for the equality $\cd(R,I)=\dim R- \depth R/I$.

\begin{lem} \label{EABR} (cf. \cite[Lemma 3.4]{EABR})
Let $(R, \fm)$ be a Noetherian local ring and $I \subset R$ an ideal of
finite projective
 dimension. Assume that $\varphi:R \rightarrow R$ is a ring  endomorphism
 satisfying the going down property and $J=\varphi(I)R$. Then $\depth R/I \leq \depth R/J$.
\end{lem}

\begin{cor} \label{depth equality}
Let $(R, \fm)$ be a Gorenstein local ring. Suppose that
\begin{enumerate}[(a)]
\item The triple $(R, I, \varphi)$ satisfies (Inj);
\item For any $t\geq 0,$ $\depth R/\varphi^{t}(I)R \leq \depth R/\varphi^{t+1}(I)R;$
\item Projective dimension of $R/I$ is finite;
%\item $\varphi$ is of finite length.
\end{enumerate}
Then, the following statements are true.
\begin{enumerate}[(1)]
\item $\depth_R R/I = \depth_R R/ \varphi^t(I)R,\ \ \text{\ for \ all \ } t \geq 0$.
\item $R/I$ is Cohen-Macaulay if and only if  $R/ \varphi^t(I)R$, is Cohen-Macaulay for all  $t \geq 0$.
\item $\Ass_R R/I = \Ass_R R/ \varphi^t(I)R,\ \ \text{\ for \ all \ } t \geq 0$.
\end{enumerate}
\end{cor}

\begin{proof}
Thanks to assumption (b) one may write $$\depth_R R/I \leq
\depth_R R/ \varphi(I)R \leq \depth_R R/ \varphi^2(I)R \leq
\cdots.$$ On the other hand, suppose that $\depth_R R/I =u$. As $(R,
I, \varphi)$ has (Inj) property, one has a monomorphism, for each $t,$
\[
\Ext_R^{\dim (R)-u} (R/\varphi^{t}(I)R,R)\hookrightarrow\Ext_R^{\dim (R)-u} (R/\varphi^{t+1}(I)R,R).
\]
By applying the Matlis duality functor $(-)^{\vee}=\Hom_R (-,E)$ to this monomorphism, one gets an epimorphism
\[
\Ext_R^{\dim (R)-u} (R/\varphi^{t+1}(I)R,R)^{\vee}\xymatrix{\ar@{->>}[r]& }\Ext_R^{\dim (R)-u} (R/\varphi^{t}(I)R,R)^{\vee}.
\]
Finally, using local duality \cite[11.2.5]{BroSha} the previous epimorphism is exactly the epimorphism
$$H^u_{\fm}(R/\varphi^{t+1}(I)R) \rightarrow H^u_{\fm}(R/\varphi^{t}(I)R),$$
of non zero modules, for all $t$. It implies that $\depth_R R/
\varphi^t(I)R=u$ for all $t.$ This proves part (1).

Now, note that $\dim_R R/I = \dim_R R/ \varphi^t(I)R$, for all $t$,
because $\{I^t\}_t$ is cofinal with $\{\varphi^t (I)R\}_t.$ In this way, part (2) follows combining part (1) jointly with Auslander--Buchsbaum formula.

Finally, note that for a prime ideal $\fp$ of $R$, $\fp \in \Ass_R M$ if
and only if $\fp R_{\fp} \in \Ass_{R_{\fp}} M_{\fp}$, where $M$ is
an $R$-module. Thus, one may assume that $\fp=\fm$. Hence, it is enough to prove that $\fm \in \Ass_R R/I$ if and only if $\fm  \in
\Ass_R R/ \varphi^t(I)R$, hence part (3) follows once again from part (1).
\end{proof}

Now is the time to prove the first main result of this section;

\begin{thrm} \label{PS Generalize} Let $(R,\fm)$ be a Gorenstein local ring.
Suppose that
\begin{enumerate}[(a)]
\item For any $t\geq 0,$ $\depth R/\varphi^{t}(I)R \leq \depth R/\varphi^{t+1}(I)R;$
\item Projective dimension of $R/I$ is finite;
%\item $\varphi$ is of finite length.
\end{enumerate}
Then one has $\cd(R,I) \leq \pd R/I$.

If in addition, the triple $(R, I, \varphi)$ satisfies (Inj)
property, then the equality holds.
\end{thrm}

\begin{proof} Suppose that $\pd R/I=u$ for some integer $u$. By the
Auslander-Buchsbaum formula $\depth R/I=\dim R-u.$ Now, combining this equality jointly with our assumption (a) one has
$$H^i_{\fm}(R/\varphi^t(I)R)=0 \text{\ for\ all\ } i < \dim R-u \text{\ and\ all\ } t.$$
This implies, again by local duality \cite[11.2.5]{BroSha}, that
\[
\Ext_R^{\dim (R)-i}(R/\varphi^t (I)R,R)^{\vee}=0\text{ for all }i<\dim (R)-u.
\]
Since Matlis duality is a faithful functor, one has that the above vanishing is equivalent to
\[
\Ext_R^{\dim (R)-i}(R/\varphi^t (I)R,R)=0\text{ for all }i<\dim (R)-u.
\]
Finally, using the fact that $\{\varphi^t (I)R\}_t$ is cofinal with respect to $\{I^t\}_t$ jointly with this Ext vanishing one obtains, for all $\dim (R)-i>u,$ that
\[
H^{\dim R-i}_I(R)=\vil \Ext_R^{\dim (R)-i} (R/\varphi^t (I)R,R)=0.
\]
This vanishing implies that $\cd(R,I) \leq \pd R/I.$
%As every composition of maps satisfying going-down property satisfies that property itself, exploiting Lemma \ref{EABR}

In order to prove the equality $\cd(R,I) = \pd R/I$, it is enough to prove $\cd(R,I) \geq \pd R/I$.
It follows from
the assumptions and Corollary \ref{depth equality} that the
homomorphisms
$$H^{\dim R-u}_{\fm}(R/\varphi^{t+1}(I)R) \rightarrow H^{\dim R-u}_{\fm}(R/\varphi^{t}(I)R),$$
induced from natural homomorphism $R/\varphi^{t+1}(I)R \rightarrow
R/\varphi^{t}(I)R$, are surjective of nonzero modules for all $t$.
Then  $ {\vpl}_t H^{\dim R-u}_{\fm}(R/I^t) \neq 0$ and $ {\vpl}_t
H^{i}_{\fm}(R/I^t) = 0$ for all $i < \dim R-u$. To this end note
that as $\{\varphi^t(I)R\}_{t \geq 0}$ and $\{I^tR\}_{t \geq 0}$ are
cofinal, one has $ {\vpl}_t H^{\dim R-u}_{\fm}(R/\varphi^t(I)R)
\cong {\vpl}_t H^{\dim R-u}_{\fm}(R/I^t)$.

Finally, combining Local Duality jointly with the fact that contravariant Hom transforms inverse limits into direct limits and that the Matlis duality functor is faithful one gets that
\begin{align*}
H_I^u (R)& =\vil \Ext_R^u (R/I^t, R)=\vil \Hom_R (H^{\dim R-u}_{\fm}(R/I^t),E)\\
& \cong\Hom_R (\vpl H^{\dim R-u}_{\fm}(R/I^t),E)\neq 0.
\end{align*}
Therefore, we have checked that $H^{u}_{I}(R) \neq
0,$ hence $u \leq \cd (R,I)$ that is $\cd(R,I) \geq \pd R/I$, as desired.
\end{proof}
The first consequence of Theorem \ref{PS Generalize} we want to single out is the below:

\begin{cor}
Let $(R,\fm)$ be a Gorenstein local ring of prime characteristic $p,$ and let $I\subseteq R$ be an ideal of finite projective dimension. Then, one has $\pd
R/I\leq \cd(R,I).$

If, in addition, $F$ denotes the Frobenius endomorphism of $R,$ and the triple $(R,I,F)$ satisfies the (Inj) property, then $\pd R/I=\cd(R,I).$
\end{cor}

\begin{proof}
Since $F$ induces identity on $\Spec (R),$ it satisfies Going Down, in particular, it increases depth thanks to Lemma \ref{EABR}. Therefore, the result follows immediately from Theorem \ref{PS Generalize}.
\end{proof}

Next, as a by-product of Theorem \ref{PS Generalize} we can recover
a result of Lyubeznik \cite{Lyu84}.

\begin{cor}\label{corollary for squarefree monomial ideals}
Let $R=k[x_1,\ldots, x_n]$ be a polynomial ring of $n$ variables
over a field $k$ and $I$ be a square-free monomial ideal. Then $\pd
R/I=\cd(R,I)$.
\end{cor}

\begin{proof}
Define the $k$-linear endomorphism $\varphi:R \longrightarrow R$ by
$x_i \mapsto x^2_i$
 for $1 \leq i \leq n.$
 It is a flat ring endomorphism and without loss of generality, after localizing at the maximal ideal $\fm=(x_1,\ldots, x_n)$ we may assume that $R$ is a local ring such that $(R,I,\varphi)$ satisfies the (GD) property. Hence, we are done
by Theorem \ref{PS Generalize}. To complete the proof note that by
what we have indicated at Remark \ref{pure2}, the endomorphism $R/I
\rightarrow R/I$ is pure.
\end{proof}

\begin{discussion}\label{Hartshorne} A special case of a conjecture of Hartshorne \cite[page 126]{Har70} is the following:

\textbf{Conjecture:} Let $C$ be a curve in $\mathbb{P}^3_k$ over a field $k$ of characteristic $0$.
 If $C$ is a set-theoretic complete intersection, the curve $C$ is arithmatically Cohen-Macaulay.
  I.e. the homogeneous coordinate ring $k[x_1, \ldots, x_n]/I(C)$ is  Cohen-Macaulay.

It is known that this conjecture is not true, see for instance
\cite{St-Vo}.  To do so further and motivated by the preceding
conjecture we consider the equivalent property between the
Cohen-Macaulayness of $R/I$ and the equality $\Ht (I) = \cd (R,I)$.
\end{discussion}

\begin{thrm} \label{CM cond} Let $(R,\fm)$ be a Gorenstein local ring.
Suppose that
\begin{enumerate}[(a)]
\item The triple $(R, I, \varphi)$ has (GD) property;
\item Projective dimension of $R/I$ is finite;
%\item $\varphi$ is of finite length.
 \end{enumerate}
 Then the following statements are equivalent.
\begin{enumerate}[(1)]
  \item $R/I$ is a Cohen-Macaulay ring.
  \item $\Ht(I)=\cd(R,I)$ and $\Ext^r_R(R/\varphi^{t}(I)R,R) \longrightarrow
  \Ext^r_R(R/\varphi^{t+1}(I)R,R)$ are injective of nonzero $R$-modules for all $t$ and $r:=\depth R/I$.
  \end{enumerate}
\end{thrm}

\begin{proof}
 $(1) \Rightarrow (2)$ First of all note that $\dim R/I= \dim R/I^t=\dim R/\varphi^{t}(I)R$
for all $t \in \mathbb{N}$. Since $R/I$ is a Cohen-Macaulay ring, by
virtue of Lemma \ref{EABR},
 $\dim R/I= \depth R/I=\depth R/\varphi^{t}(I)R $ for all $t \in \mathbb{N}$. The
 first part follows from Theorem \ref{PS Generalize} and the Auslander-Buchsbaum Theorem. For the second part consider
  the short exact sequence
  $$0 \rightarrow \varphi^{t}(I)R/\varphi^{t+1}(I)R \rightarrow
   R/\varphi^{t+1}(I)R \rightarrow R/\varphi^{t}(I)R \rightarrow 0.$$ Applying $H^{i}_{\fm}(-)$ and
    use the Grothendieck's Vanishing Theorem and the fact that
    $\dim \varphi^{t}(I)R/\varphi^{t+1}(I)R \leq r$ we obtain an
     epimorphism $H^r_{\fm}(R/\varphi^{t+1}(I)R) \longrightarrow H^r_{\fm}(R/\varphi^t(I)R)$. Now, by local duality \cite[11.2.5]{BroSha} this surjection is equivalent to the surjection
\[
\Ext_R^{\dim (R)-r}(R/\varphi^{t+1}(I)R,R)^{\vee}\xymatrix{\ar@{->>}[r]& }\Ext_R^{\dim (R)-r}(R/\varphi^{t}(I)R,R)^{\vee}.
\]
Applying to this surjection the Matlis duality functor $(-)^{\vee}$ one gets an injection
\[
\Ext_R^{\dim (R)-r}(R/\varphi^{t}(I)R,R)^{\vee\vee}\hookrightarrow\Ext_R^{\dim (R)-r}(R/\varphi^{t+1}(I)R,R)^{\vee\vee}.
\]
Now, consider the commutative square, where the vertical arrows are the natural homothety maps from a module to its Matlis bidual, and the horizontal arrows are the natural ones:
\[
\xymatrix{\Ext_R^{\dim (R)-r}(R/\varphi^{t}(I)R,R)\ar[d]\ar[r]& \Ext_R^{\dim (R)-r}(R/\varphi^{t+1}(I)R,R)\ar[d]\\
\Ext_R^{\dim (R)-r}(R/\varphi^{t}(I)R,R)^{\vee\vee}\ar[r]& \Ext_R^{\dim (R)-r}(R/\varphi^{t+1}(I)R,R)^{\vee\vee}.}
\]
Notice that the vertical maps are injective, as seen in \cite[10.2.1 and 10.2.2]{BroSha}; moreover, the bottom one is also injective by what we have seeen before, hence the top one is also injective, which is exactly what we wanted to prove.

$(2) \Rightarrow (1)$ By the assumptions $\Ht(I)=\cd(R,I)$, i.e.
$H^i_I(R)=0$ for all $i \neq \Ht(I)$. By local duality one has
\begin{eqnarray*}
 \Hom_R(H^i_I(R),E(R/\fm))  & \cong &\Hom_R(\vil \Ext^i_R(R/I^t,R),E(R/\fm)) \\
   &\cong &  \vpl \Hom_R(\Ext^i_R(R/I^t,R),E(R/\fm))  \\
   &\cong &  \vpl H^{\dim R-i}_{\fm}(R/I^t).
\end{eqnarray*}
It means that for all $i \neq \Ht (I)$,  $ \vpl H^{\dim
R-i}_{\fm}(R/I^t) = 0$. Now, it is enough to prove that $ \vpl
H^{r}_{\fm}(R/I^t) \neq 0$. This follows from the fact that the
homomorphisms
$$H^r_{\fm}(R/\varphi^{t+1}(I)R) \longrightarrow
H^r_{\fm}(R/\varphi^t(I)R)$$ are surjective  of nonzero $R$-modules
for all $t$.
\end{proof}
As immediate consequence of Theorem \ref{CM cond} we obtain the following statement, which recovers and extends \cite[Proposition 3.2]{Eghbali2014}:

\begin{cor}\label{CM cond on prime characteristic}
Let $(R,\fm)$ be a Gorenstein local ring of prime characteristic $p,$ and let $I\subseteq R$ be an ideal of finite projective dimension. Then, we have that $R/I$ is Cohen--Macaulay if and only if $\Ht(I)=\cd(R,I)$ and $\Ext^r_R(R/I^{[p^t]},R) \longrightarrow
  \Ext^r_R(R/I^{[p^{t+1}]},R)$ are injective of nonzero $R$-modules for all $t$ and $r:=\depth R/I$.
\end{cor}

\begin{thrm}\label{vanishing result of local cohomology}
Let $(R,\mathfrak{m})$ be a Gorenstein local ring, let $I\subset R$ be an ideal, and let $\xymatrix@1{R\ar[r]^-{\varphi}& R}$ be a flat ring endomorphism such that $\{\varphi^t(I)R\}_t$ is a descending chain of ideals that is cofinal with $\{I^t\}_t,$ and $\{\varphi^t(\mathfrak{m})R\}_t$ is a descending chain of ideals that is cofinal with $\{\mathfrak{m}^t\}_t.$ Then, for any integer $j\geq 0,$ $H_I^j (R)=0$ if and only if there is an integer $t\gg 0$ such that the map
\[
H_{\mathfrak{m}}^{\dim(R)-j}(R/I)\xymatrix{\ar[r]^-{\varphi^t}& }H_{\mathfrak{m}}^{\dim(R)-j}(R/I)
\]
is zero.
\end{thrm}
%We also assume that $\Hom_R (\varphi_*^t R , E)\cong E$ for any $t\geq 0,$ where $E$ denotes the injective hull of the residue field of $R$ (this holds, e.g., when $\varphi_* (R)$ is a finitely generated $R$-module \cite[Lemma 3.7]{Marley2014})

\begin{proof}
Since $\varphi$ is flat and $\{\varphi^t(I)R\}_t$ is cofinal with $\{I^t\}_t,$ we have
\[
H_I^j (R)=\vil \Ext_R^j (R/\varphi^t(I)R,R)=\vil \Phi^t\left(\Ext_R^j (R/I,R)\right).
\]
The above equality shows that $H_I^j (R)=0$ if and only if there is an integer $t\gg 0$ such that the natural map
\[
\Ext_R^j (R/I,R)\xymatrix{\ar[r]& }\Phi^t\left(\Ext_R^j (R/I,R)\right)
\]
Now, we apply the Matlis duality functor to this zero map; indeed, on the one hand the source of the map becomes $H_{\mathfrak{m}}^{\dim(R)-j}(R/I)$ just because of Local duality. On the other hand, using the isomorphism $\Phi(E)\cong E$ that is given by the assumption that $\{\varphi^t(\mathfrak{m})R\}_t$ is cofinal with $\{\mathfrak{m}^t\}_t,$ jointly again with Local Duality one can check that the target is
\begin{align*}
& \Phi^t\left(\Ext_R^j (R/I,R)\right)^{\vee}=\Hom_R (\Phi^t\left(\Ext_R^j (R/I,R)\right),\Phi^t E)\\& \cong\Phi^t\Hom_R (\Ext_R^j (R/I,R),E)\cong \Phi^t H_{\mathfrak{m}}^{\dim(R)-j}(R/I).
\end{align*}
Summing up, the above zero map implies (check) that the map $\varphi^t$ on $H_{\mathfrak{m}}^{\dim(R)-j}(R/I)$ is also zero, which is what we finally wanted to prove.
\end{proof}

\begin{remark}
Notice that the proof of the above Theorem is just the one sketched in \cite[Theorem 22.1]{Twentyfourhours} for the case of the Frobenius map in prime characteristic in a regular local ring, see \cite[Theorem 1.1]{Lyu06} for full details. It can be regarded as a mild generalization of \cite[Theorem 4.1]{S-W}.
\end{remark}

\section{SET-THEORETICALLY COHEN-MACAULAY IDEALS}\label{set theoretically C-M ideals}

It is known that the radical of a Cohen-Macaulay ideal need not to
be Cohen-Macaulay in general. A well-known evidence is an example
due to Hartshorne \cite{Ha} shows that whenever $k$ is a field of
positive prime characteristic, the ideal $$I= \ker (\varphi: k[x_1, x_2,
x_3, x_4] \rightarrow k[s^4, s^3t,st^3, t^4])$$ via $x_1 \mapsto
s^4, x_2 \mapsto s^3t, x_3 \mapsto st^3, x_4 \mapsto t^4$ is a
non-Cohen-Macaulay set-theoretic complete intersection. However,
there are some Cohen-Macaulay ideals having the same property for
their radicals. For instance, if $I$ is a Cohen-Macaulay monomial
ideal of a polynomial ring $R = k[x_1, \ldots, x_n]$ over a field
$k$, then $\sqrt{I}$ is Cohen-Macaulay (cf. \cite[Theorem
2.6]{HTT}). Furthermore, principal and Veronese monomial ideals have
such a property, \cite[Theorem 3.2]{HH}.

Let us recall the following definition in order to strengthen the
above results.

\begin{definition}  Let $(R,\fm)$ be a regular local ring. An ideal $I \subset R$ is
called set-theoretically Cohen-Macaulay if there exists an ideal $J
\subset R$ with $\sqrt{I} = \sqrt{J}$ such that the ring $R/J$ is
Cohen-Macaulay.
\end{definition}

It is clear that Cohen-Macaulay radical ideals are set-theoretically
Cohen- Macaulay. It should be noted that there exist also
set-theoretic Cohen-Macaulay ideals which are not Cohen-Macaulay.
Suppose $R=k[\![x,y]\!]$ is a formal power series ring over a field $k$
of $x,y$. Put $I=(y) \cap(xy,y^2)$ and $J=\sqrt{I}$. It is clear that
$R/J$ is Cohen-Macaulay but this is not the case for $R/I$.

Our first result in this direction is a consequence of Theorem
\ref{CM cond} stating that under the assumptions given in
Proposition \ref{equivalence} the concepts of Cohen-Macaulayness and
set-theoretically Cohen-Macaulayness are the same.

\begin{prop} \label{equivalence} Let $(R,\fm)$ be a regular local ring with a flat ring endomorphism $\varphi: R \rightarrow
R$. Further suppose that triples $(R,I,\varphi)$ and $(R,J,\varphi)$
have (Inj) property where $\sqrt{I}=\sqrt{J}$. Then $R/I$ is
Cohen-Macaulay if and only if $R/J$ is Cohen-Macaulay.
\end{prop}

\begin{proof} First note that as $(R,\fm)$ is a regular local ring with a flat local ring endomorphism $\varphi: R \rightarrow
R$, the dimension formula $\dim R + \dim R/\varphi(\fm)R = \dim R$
implies that $\varphi(\fm)R$ is $\fm$-primary. Then, since $R$ is a
Noetherian ring and $\sqrt{I}=\sqrt{J},$ thus $\{I^t\}_{t \geq 0}$ and
$\{J^t\}_{t \geq 0}$ are cofinal. Now we are done by Theorem \ref{CM cond}.
\end{proof}

As non trivial consequence of Proposition \ref{equivalence}, we may recover \cite[Lemma 3.1]{S-W2}.

\begin{cor} (Huneke) Let $(R,\fm)$ be a regular local ring of positive characteristic $p>0$ and $I$ an
 ideal of $R$. If the ring  $R/I$ is $F$-pure, and $I$ is set--theoretically Cohen--Macaulay, then $I$ is a Cohen-Macaulay ideal.
\end{cor}

\begin{proof}
Since $I$ is a set-theoretically Cohen-Macaulay ideal, there is an ideal $J\subseteq R$ with $\sqrt{J}=\sqrt{I}$ such that $R/J$ is Cohen--Macaulay. Now, denoting by $F$ the Frobenius endomorphism of $R,$ on the one hand the triple $(R,I,F)$ has (Inj) property because $R/I$ is $F$--pure by assumption; on the other hand, the triple $(R,J,F)$ also has (Inj) property as consequence of Corollary \ref{CM cond on prime characteristic}. Keeping in mind all these facts, we conclude, thanks to Proposition \ref{equivalence}, that $I$ is Cohen--Macaulay.
\end{proof}

 According to Discussion \ref{Hartshorne} in Section \ref{on a result of Peskine-Szpiro}, we observe that
Cohen-Macaulayness of $R/I$ is not
  equivalent to the ideal $I$ being set-theoretically complete intersection. On the other
  hand, set-theoretically complete intersection ideals are
  set-theoretically Cohen-Macaulay but the converse is no longer true. See for instance \cite{S-W2}.
 We show that instead of set-theoretically complete intersection ideals one may regard cohomologically
 complete intersection ideals.

Next result (Theorem \ref{setCM-setCI}) shows the relation between
set-theoretically Cohen-Macaulay and cohomologically complete
intersection ideals. Notice that a set-theoretically complete
intersection ideal is a cohomologically complete intersection one
but the converse is no longer true.

%\begin{lem} \label{cofinal}  Let $(R,\fm)$ be a local ring and $I$ , $J$ two ideals of $R$ with a ring endomorphism  $\varphi: R \rightarrow R$ such that $\sqrt{I}=\sqrt{J}.$ Then $\sqrt{\varphi(I)R}=\sqrt{\varphi(J)R}.$
%\end{lem}

%\begin{proof}
%In order to prove $\sqrt{\varphi(I)R}=\sqrt{\varphi(J)R}$ suppose that $x$ is an arbitrary element in $\sqrt{\varphi(I)R}.$ Then there exist an integer $n$ and an element $a \in I$ such that $x^n=\varphi(a)$. As $\sqrt{I}=\sqrt{J},$ there is an integer $t$ such that $a^t \in J$. Therefore, $x^{nt}=\varphi(a^t) \in\varphi(J)R$, that is $x \in\sqrt{\varphi(J)R}.$

%For the reverse inclusion, from $x\in\sqrt{\varphi(J)R}$ we observe that $x^n=\varphi(a)$ for some $a \in J$ and $n \in \mathbb{Z}$. Once again from the equality $\sqrt{I}=\sqrt{J}$ one has $a^t \in I$ for some integer $t$. By applying $\varphi(-)$ to it we have $x^{nt}=\varphi( a^t) \in \varphi(I)R$. It completes the proof.
%\end{proof}

\begin{thrm} \label{setCM-setCI} Let $(R,\fm)$ be a  regular local ring and the triple $(R, I, \varphi)$ has (Flat) property.
 Suppose that $I$ is a set-theoretically Cohen-Macaulay
 ideal such that there is an ideal $J\subseteq R$ with $\sqrt{I}=\sqrt{J}$ such that $R/J$ is Cohen-Macaulay and $\{\varphi^t (J)R\}_t$ is cofinal with respect to $\{I^t\}_t.$ Then, one has $\Ht(I)=\cd (R,I)$.
\end{thrm}

\begin{proof} Preserving the assumptions and notations established above, we know that
$\sqrt{I}=\sqrt{J}$ and $H^i_{\fm}(R/J)=0$ for all $i \neq \dim R/I$.
As $\varphi:R \rightarrow R$ is a flat ring endomorphism, then so does all its iterations. By applying the functor $\Phi^t$ to $H^i_{\fm}(R/J)$, one has $H^i_{\fm}(R/\varphi^t(J)R)=0$ for all integer $t$ and all $i \neq \dim R/I.$ This implies, by local duality, that
\[
\Ext_R^{\dim (R)-i} (R/\varphi^t(J)R,R)^{\vee}=0\text{ for all }\dim (R)-i\neq\dim (R)-\dim (R/I),
\]
hence $\Ext_R^{\dim (R)-i} (R/\varphi^t(J)R,R)=0$ for the same range of values because of the faithfulness of Matlis duality. Finally, combining this vanishing of Ext's modules jointly with the fact that $\{\varphi^t (J)R\}_t$ is cofinal with respect to $\{I^t\}_t$ one has 
\[
H^j_{I}(R)=\vil_t \Ext_R^j (R/\varphi^t (J)R,R)=0\text{ for all }j \neq \Ht (I)
\]
It implies that $\Ht(I)=\cd(R,I).$
\end{proof}
%By the assumption $(R, I, \varphi)$ has (Flat) property and Lemma \ref{cofinal} implies that $\sqrt{J}=  \sqrt{I}= \sqrt{\varphi (I)R}= \sqrt{\varphi (J)R}.$

Next result, which is a non--completely obvious consequence of Theorem \ref{setCM-setCI} and Theorem \ref{PS Generalize}, recovers and extends \cite[Corollary 4.3]{Eghbali2014}.

\begin{cor} \label{equi.setCM-setCI} Let $R=k[x_1, \ldots, x_n]$ be a polynomial ring over a field $k$
 and let $\fm=(x_1, \ldots, x_n)$ be the maximal ideal.
 Suppose that $I$ is an ideal such that $\sqrt{I}$ is a square free monomial ideal.
 Then the following statements are equivalent.
 \begin{enumerate}[(a)]
\item $I$ is a set-theoretically Cohen-Macaulay ideal;
\item $I$ is a cohomologically complete intersection ideal.
\end{enumerate}
\end{cor}

\begin{proof}
Let $\varphi$ be the $k$--algebra endomorphism that maps any variable $x_i$ to its square, which is a flat map; moreover, we also know that, for any ideal $J\subseteq R,$ $\{\varphi^t (J)R\}$ is cofinal with respect to $\{J^t\}_t.$ Now, we can appeal to Theorem \ref{setCM-setCI} to conclude that (a) implies (b).

To prove the reverse statement, it is enough to show that
$H^i_{\fm}(R/\sqrt{I})=0$ for all $i \neq \dim R/I$. First of all, notice that, since the map $\varphi$ induced on $R/\sqrt{I}$ is pure, one has that the triple $(R,\sqrt{I},\varphi)$ satisfies the (Inj) property (cf. Remark \ref{pure2}). In this way, the result follows immediately from Theorem \ref{PS Generalize}.
\end{proof}

The equivalence between the cohomologically complete intersection
and set-theoretically Cohen-Macaulay ideals in Theorem
\ref{setCM-setCI} holds whenever $R/\sqrt{I}$ is pure and flat.

We conclude this section with an investigation on the cohomological
dimension of the intersection of set-theoretically Cohen-Macaulay
ideals. Before it, we need the following Lemma.

\begin{lem} \label{intersec} Let $(R,\fm)$ be a regular local ring and $I$, $J$
 two ideals of $R$ with a flat ring endomorphism  $\varphi: R \rightarrow R$
  such that $\sqrt{I}=\sqrt{\varphi(I)R}$ and $\sqrt{J}=\sqrt{\varphi(J)R}.$
  Then $\sqrt{I \cap J}=\sqrt{\varphi(I \cap J)R}.$
\end{lem}

\begin{proof}
To prove, it is enough to note that by the flatness of $\varphi$,
one has $\varphi(I \cap J)R=\varphi(I)R \cap \varphi(J)R$.
\end{proof}

\begin{prop}\label{an upper bound of cohomological dimension}
Let $(R,\fm)$ be a  regular local ring and the triples $(R,I,\varphi)$
and $(R,J,\varphi)$ have  (Flat) property where $\sqrt{I},\ \sqrt{J}$
are Cohen-Macaulay ideals with the same dimension $d$. Further
suppose that $\sqrt{I} \cap \sqrt{J} =\sqrt{I} \sqrt{J}$ and $d >\depth
R/(\sqrt{I}+\sqrt{J}).$  Then $\cd (R, I \cap J) \leq
\Ht(I)+\Ht(J)-1.$
\end{prop}

\begin{proof} As cohomological dimension is stable under taking radical of ideals, Lemma \ref{intersec} and Theorem \ref{PS Generalize} imply that
$$\cd (R, I \cap J) =\cd (R,\sqrt{I} \cap \sqrt{J}) \leq \dim R -
\depth R/\sqrt{I} \cap \sqrt{J}.$$
Now, we claim that
$$\depth (R/\sqrt{I} \cap \sqrt{J} ) \geq  \depth (R/(\sqrt{I} + \sqrt{J}))+1.$$
Indeed, by \cite[Lemma 8.7 (2)]{Twentyfourhours} we know that
\[
\depth (R/\sqrt{I} \cap \sqrt{J} )\geq\min\{\depth(R/\sqrt{I}\oplus R/\sqrt{J}),\depth (R/\sqrt{I} + \sqrt{J})+1\}.
\]
Moreover, using also \cite[Lemma 8.7 (1)]{Twentyfourhours} we also have
\[
\depth (R/\sqrt{I} \cap \sqrt{J} )\geq\min\{\depth(R/\sqrt{I}),\depth(R/\sqrt{J}),\depth (R/\sqrt{I} + \sqrt{J})+1\}.
\]
From this last upper inequality we deduce, because of the assumption $d >\depth
R/(\sqrt{I}+\sqrt{J}),$ that
\[
\depth (R/\sqrt{I} \cap \sqrt{J} ) \geq  \depth (R/(\sqrt{I} + \sqrt{J}))+1,
\]
and therefore
\begin{equation}\label{1a}
   \cd (R, I \cap J) \leq \dim R - \depth (R/\sqrt{I} +\sqrt{J}) -1.
\end{equation}
On the other hand, our assumption $\sqrt{I} \cap \sqrt{J} =\sqrt{I} \sqrt{J}$ implies that $\Tor^R_1(R/\sqrt{I}, R/\sqrt{J})=0$.
 By rigidity (see either \cite[Corollary 1]{L} or \cite[Corollary 2.5]{CelikbasWiegand2015}) $\Tor^R_i(R/\sqrt{I}, R/\sqrt{J})=0$ for all $i \geq 1$.
  Hence, by \cite[Theorem 1.2]{Aus} the depth formula holds for $(R/\sqrt{I}, R/\sqrt{J}).$ Therefore, one has that
\[
\depth (R/\sqrt{I} + \sqrt{J})=\depth ((R/\sqrt{I})\otimes_R (R/\sqrt{J}))=\depth (R/\sqrt{J})-\pd (R/\sqrt{I}).
\]
Combining this equality jointly with Auslander--Buchsbaum formula for $\pd (R/\sqrt{I})$ one gets
\begin{equation}\label{2a}
  \depth (R/\sqrt{I} + \sqrt{J}) = -\dim R+ \depth (R/\sqrt{I})+ \depth (R/\sqrt{J}).
\end{equation}
 Combining (\ref{1a}) and (\ref{2a}) we conclude that $\cd (R, I \cap J) \leq
\Ht(I)+\Ht(J)-1$.
\end{proof}
We conclude this section with the following:

\begin{remark}
We would like to mention that the upper bound of the cohomological dimension of the intersection of two ideals is far from being sharp, the interested reader may like to consult \cite[Theorem 1.1 and Corollary 1.2]{Lyubeznik2007}, \cite[Theorem 3.8]{DaoTakagi2016}, and the references given therein for further details. For lower bounds, one can consult, for instance, \cite[19.2.8]{BroSha}.
\end{remark}

\section{LINKAGE}\label{section about linkage}

 In  the present section, we consider the concept of linkage. Roughly speaking, the linkage is the study of two subschemes where their
 union has nice properties. Exploiting this concept one may consider a subscheme linked with the second one which one understands better.
To be more precise and from local algebra point of view we recall the definition of two linked ideals.

\begin{definition} Let $I$ and $J$ be two ideals of pure height $g$ of a local Gorenstein ring $(R,\fm)$.
The ideals $I$ and $J$ are (algebraically) linked by a complete intersection
$\underline{x}:=x_1, \ldots, x_g$ with $(\underline{x}) \subseteq I \cap J$  if
 $I/(\underline{x}) \cong \Hom_R(R/J, R/(\underline{x}))$ and $J/(\underline{x}) \cong \Hom_R(R/I, R/(\underline{x}))$.
 We write it as $I \sim_{(\underline{x})} J$. For brevity we often write $I \sim J$ for $I \sim_{(\underline{x})} J$
  when there is no ambiguity about the ideal $(\underline{x})$.
\end{definition}

Ideals $I$ and $J$ are in the same linkage class if there is a
sequence of links $J=I_0 \sim I_1 \sim \ldots \sim I_q= I$. An ideal
$I$ is in the linkage class of a complete intersection if any of the
ideals in the linkage class of $I$ are generated by a regular
sequence. If $q$ is an even integer, we say that $J$ is in the even
linkage class of $I$.

Let $I$ and $J$ be two ideals  linked by a complete intersection $\fc$ of height $g$. Then
$$\Hom_R(R/I, R/\fc) \cong \Hom_{R/\fc}(R/I, R/\fc) \cong \Ext^g_R(R/I,R):=K_{R/I}$$
and
 $$\Hom_R(R/J, R/\fc) \cong \Hom_{R/\fc}(R/J, R/\fc) \cong \Ext^g_R(R/J,R):=K_{R/J}$$
are the canonical modules of $R/I$ resp. $R/J$.

A ring $R$ satisfies Serre's condition $(S_r)$ if for all $\fp \in
\Spec R$,
$$\depth R_{\fp} \geq \min \{r, \dim R_{\fp} \}.$$
For an $n$-dimensional ring, being Cohen-Macaulay is equivalent to
satisfying $(S_n)$.

\begin{lem} \label{Serr charact}  \cite[Theorem 1.14] {Sch98} Let $M$ denote a finitely generated, equidimensional $R$-module with $d = \dim M$, where $R$ is a factor ring of a Gorenstein ring. Then
for an integer $r \geq 1$ the following statements are equivalent:
\begin{enumerate}
   \item[(a)]  $M$ satisfies condition $(S_r)$.

  \item[(b)]  The natural map $M \rightarrow K_{K_M}$ is bijective (resp. injective for $r = 1$) and
  $$H^i_{\fm}(K_M) = 0 \text{ \ for\  all\ } d - r + 2 \leq i < d.$$
\end{enumerate}
\end{lem}

We will use the following facts on cohomological relations of linked
ideas, later on.

\begin{lem} \label{Schenzel}  Let $I$ and $J$ be two linked ideals of a local Gorenstein ring $(R,\fm)$.
Suppose that $E:=E(R/\fm)$ is the injective hull of $R/\fm$ and
$d=\dim R/I=\dim R/J$.
\begin{enumerate}
   \item[(a)]  There exist a canonical exact sequence
$$0 \rightarrow H^{d-1}_{\fm}(R/J) \rightarrow H^{d}_{\fm}(K_{R/I}) \rightarrow \Hom_R(R/I,E) \rightarrow 0,$$
 and the canonical isomorphisms
$$H^{i-1}_{\fm}(R/J) \cong H^{i}_{\fm}(K_{R/I}),\ i < d,$$
(cf. \cite[Lemma 4.2]{Sch82}).

  \item[(b)]  For an integer $r \geq 2$ the
following statements are equivalent:
\begin{enumerate}
   \item[(1)] $R/I$  satisfies $(S_r)$;
  \item[(2)] $H^{i}_{\fm}(R/J)=0$ for all integers $d-r<i<d$,
  (cf.
\cite[Theorem 4.1]{Sch82}).
\end{enumerate}

\item[(c)]  Assume that the local cohomology
modules $H^{i}_{\fm}(R/I)$ have finite length over $R$ for all
integers $i=0, 1,\ldots, d-1$. Then,
 $$H^{d-i}_{\fm}(R/J)=\Hom_R(H^{i}_{\fm}(R/I),E),\  \text{\ for\ all\ integers\ } i=1,\ldots, d-1,$$
 (cf. \cite[Theorem 1.2, pp. 157]{St-Vo} or \cite[Corollary 3.3]{Sch82}).
\end{enumerate}
\end{lem}

\begin{remark} \label{ass}  Let  $I, J, \fc$ be as before. Then $\fm \nsubseteq \Ass (R/I)$ if and only if $\fm \nsubseteq \Ass
(R/J)$. To see this, without loss of generality assume that $R/I$ is
not Cohen-Macaulay. Suppose that $\fm \nsubseteq \Ass (R/I)$. From
the above descriptions, one can obtain the following exact sequence
\begin{equation}\label{exact}
   0\longrightarrow K_{R/J}\longrightarrow R/\fc\longrightarrow
R/I \longrightarrow0.
\end{equation}
As $\zd(R/I)=\bigcup_{\fp \in \Ass_R(R/I)} \fp$, we observe that
$\fm \nsubseteq \zd(R/I)$ ($\fm$ is not contained in the zero
divisors of $R/I$), that is $H^{0}_{\fm}(R/I)=0$. From the exact
sequence (\ref{exact}), one has $\depth K_{R/J}-1=\depth R/I$ (for
detailed proof see \cite[Proposition 3.3]{egh-sh}). It implies that
$\depth K_{R/J}
>1$. once again, in the light of (\ref{exact}) we have
$H^{0}_{\fm}(R/\fc)=0$. Hence, we conclude that $\fm \nsubseteq \Ass
(R/J)$. To this end note that $ \Ass (R/I) \cup \Ass (R/J)= \Ass
(R/\fc)$.
\end{remark}

Of particular interest is that if $I$ and $J$ are in the same
linkage class then what properties of $I$ are shared by $J$. For
instance, in \cite[Proposition 1.3]{P-S74} Peskine and Szpiro proved
that if $I$ and $J$ are linked ideals and $I$ is a Cohen-Macaulay
ideal then so is $J$. The same result was proved by Schenzel in
\cite[Corollary 3.3]{Sch82} for Buchsbaum ideals. In this direction,
using the results of Schenzel (Lemma \ref{Schenzel}), with some mild
assumptions, we show that between evenly linked ideals the Serre's
condition $(S_r)$ can be shared.

\begin{prop} \label{even sr} Let $(R,\fm)$ be a  Gorenstein local ring and $r \geq 2$ be an integer. Suppose
that either
\begin{enumerate}
   \item[(a)] $I, J$ are two evenly linked ideals of $R$, or
  \item[(b)] $I, J$ are two linked ideals of $R$ and $\fm \nsubseteq
\Ass(R/I)$.
\end{enumerate}
Then, $R/I$ satisfies $(S_r)$ if and only if $R/J$ is so.
\end{prop}

\begin{proof} (a) Suppose that $\fa$ is an ideal of $R$ where $I \sim \fa \sim J$. By virtue of Lemma \ref{Serr
charact},
$R/I$ satisfies $(S_r)$ if and only if
\begin{equation}\label{sr1}
   H^d_{\fm}(K_{R/I}) \rightarrow \Hom_R(R/I,E(R/\fm)),\ d= \dim
   R/I,
\end{equation}
 is bijective  and $H^i_{\fm}(K_{R/I})=0,\ d-r+1 < i < d$.
As $I$ is linked to $\fa$, by Lemma \ref{Schenzel}(a)
 $$0=H^{i-1}_{\fm}(R/\fa) \cong H^{i}_{\fm}(K_{R/I}),\ d-r+1 < i < d,$$
 and as $\fa$ is linked to $J$,
 $$H^{i-1}_{\fm}(R/\fa) \cong H^{i}_{\fm}(K_{R/J}),\ d-r+1 < i < d.$$
 From the exact sequence (\ref{sr1}) and Lemma \ref{Schenzel}(a) one has $H^{d-1}_{\fm}(R/\fa)=0$. Thus, by the following exact sequence
 $$0 \rightarrow H^{d-1}_{\fm}(R/\fa) \rightarrow H^{d}_{\fm}(K_{R/J}) \rightarrow \Hom_R(R/J,E(R/\fm)) \rightarrow 0,$$
 one has the homomorphism
 $H^d_{\fm}(K_{R/J}) \rightarrow \Hom_R(R/J,E(R/\fm))$ is bijective. Once again using Lemma \ref{Serr charact}
 it turns out that $R/J$ satisfies $(S_r)$.

 (b) First of all note that as $\fm \nsubseteq
\Ass(R/I)$, by Remark \ref{ass} one has $\fm \nsubseteq \Ass(R/J)$,
that is $H^0_{\fm}(R/I)=0=H^0_{\fm}(R/J)$. Then, we are done by
virtue of Lemma \ref{Schenzel}(b),(c).
\end{proof}

Of particular interest is examining cohomological dimension of two
linked ideals.

\begin{prop} \label{evevanish} Let $(R,\fm)$ be a  Gorenstein local ring and $I, J$ two evenly linked ideals of $R$.
 Suppose that
 \begin{enumerate}
   \item[(a)] the triples $(R,I, \varphi)$ and $(R,J, \varphi)$ have (Flat) property;
\item[(b)] the triple $(R,I, \varphi)$ has (Inj)
 property;
  \item[(c)] $\varphi$ is of finite length.
\end{enumerate}
Then for a given integer $i$, $H^{i}_{I}(R)=0$ if and
only if $H^{i}_{J}(R)=0.$
\end{prop}

\begin{proof}  Note that two linked ideals have the same height so for $i \leq \Ht (I)$ we have nothing to prove.
 Hence, we may assume that $i > \Ht (I)$.

Suppose that $\fa$ is an ideal of $R$ where $I \sim \fa \sim J$. By
what we have seen in the proof of Proposition \ref{even sr}, in the
light of Lemma \ref{Schenzel}(a) one has $H^{j}_{\fm}(R/I)\cong H^{j}_{\fm}(R/J)$
for all $j < \dim R/I$. Hence, the claim follows
from Theorem \ref{vanishing result of local cohomology}.
\end{proof}

\begin{remark} \label{skew lines} It is noteworthy to mention that the Proposition \ref{evevanish} is
no longer true for non-evenly linked ideals. To see this, let
$R=k[x_0,x_1,x_2,x_3]$ be a polynomial ring over an algebraically
closed field $k$. Let $I=(x_0,x_1)\cap(x_2,x_3)$ be the defining
ideal of the union of the two skew lines in ${\mathbb P}^3$ and
$J=(x_0x_3-x_1x_2,x^3_1-x^2_0x_2,x_0x^2_2-x^2_1x_3,x^3_2-x_1x^2_3)$
be the defining ideal of the  twisted quartic curve in ${\mathbb
P}^3$. Then it is not hard to show that
$$I \cap J=(x_0x_3-x_1x_2,x_0x_2^2-x_1^2x_3)$$
is a complete intersection. Therefore $I$ is linked to $J$ by
$\fc:=I \cap J$, where $H^{3}_{I}(R)\neq 0$
 and $H^{3}_{J}(R)= 0$.
\end{remark}

Next, we consider the property of being set-theoretically
Cohen-Macaulay between linked ideals.

\begin{thrm} \label{link1} Let $(R,\fm)$ be a regular local ring and let $I, J$ be two linked ideals of $R$.
 Suppose that
\begin{enumerate}[(a)]
   \item $(R,I, \varphi)$  satisfies (Flat) property,
  \item $(R,I, \varphi)$  satisfies (Inj) property,
  \item  $\fm \nsubseteq \zd(R/I)$.
\end{enumerate}
If there exists a Cohen--Macaulay ideal $\fb$ with $\sqrt{b}=\sqrt{I}$ such that $\{\varphi^t (\fb)R\}$ is cofinal with respect to $\{\fb^t\},$ then $J$ is
Cohen-Macaulay.
\end{thrm}

\begin{proof}
Note that,  $d:=\dim R/\varphi(I)R=\dim R/I=\dim R/J,$ where
 the third equality follows by the fact that two linked ideals have the same
  dimension. We are going to show that $H^{i}_{\fm}(R/J)=0$ for all $i < d$.

Let $\fb$ be the ideal described in the assumptions. Then $H^i_{\fm}(R/\fb)=0$ for all $i < d.$ It then follows from Theorem \ref{vanishing result of local cohomology} that $H^{\dim R-i}_{I}(R) \cong H^{\dim R-i}_{\fb}(R)=0$ for all $\dim R-i > \dim R- d$.
Now, using Theorem \ref{vanishing result of local cohomology} again shows that $H^{i}_{\fm}(R/I)=0$
 for all $i <d$. Hence, Lemma \ref{Schenzel}(c) ensures that $H^{i}_{\fm}(R/J)=0$ for $i=1, \ldots, d-1$.
 In this way, combining the previous information jointly with Remark \ref{ass} it follows that
 $H^0_{\fm}(R/J)=0$, this completes the proof.
\end{proof}

\begin{cor} \label{setlink} Let $(R,\fm)$ be a regular local ring and let $I, J$ be two linked ideals of
$R$. Suppose that
\begin{enumerate}[(a)]
\item $(R,I, \varphi)$ and $(R,J, \varphi)$ satisfy (Flat) property,

\item $\fm \nsubseteq \zd(R/I)$,

\item the induced ring endomorphisms $R/\sqrt{I} \rightarrow R/\sqrt{I}$ and $ R/\sqrt{J} \rightarrow R/\sqrt{J}$ are pure and flat.
\end{enumerate}
  Then,  if $I$ is set-theoretically Cohen-Macaulay then so is $J$.
\end{cor}

\begin{proof}  Theorem \ref{link1} ensures that $J$ is a Cohen-Macaulay ideal.
Then, the equality $\Ht(J)= \cd(R,J)$ holds by Theorem \ref{CM cond}. Now, exploiting
 Theorem \ref{setCM-setCI} and the proof of Corollary \ref{equi.setCM-setCI}, one observe that $J$ is a set-theoretically Cohen-Macaulay ideal.
\end{proof}
Notice that one can not remove the purity assumption from the above
result, as the following example shows.

\begin{remark}  Let $R,\ I$ and $J$ be as in Remark \ref{skew lines}.
Since $\depth R/I=1$ and $\dim R/I=2$, then $I$ is a
 non-Cohen-Macaulay radical ideal. As $I$ is a square-free ideal, $R/I$ is
  $F$-pure but this is not the case for $R/J$, because it is not reduced.
  On the other hand by virtue of Hartshorne \cite{Ha} the ideal $J$ is
  set-theoretically Complete intersection and then it is set-theoretically Cohen-Macaulay.
 \end{remark}

%%%%%%%%%%%%%%%%%%%%%%%%%%%%%%%%%%%%%%%%%%%%%%%%%%%

\end{document}